\documentclass[12pt,a4paper]{article}
\usepackage{amsmath}
\usepackage{amssymb}
\usepackage{amsthm}
\usepackage{amsfonts}
\usepackage{mathscinet}
\usepackage{enumerate}
\usepackage{pdfsync}
\usepackage{setspace}
\usepackage{mathrsfs}
\usepackage{tikz}
\usepackage{stmaryrd}
\usepackage{hyperref}
\usepackage{authblk}
\usetikzlibrary{matrix}

\usepackage[all]{xypic}

\newcommand \lh {\text{lh}}

\newcommand {\new} {\newcommand}
\newcommand {\renew} {\renewcommand}
\newcommand \opera[2] {\renew #1 {\operatorname{#2}}}
\newcommand \oper[2] {\new #1 {\operatorname{#2}}}

\newcommand \gcode[1] {\ulcorner\! #1 \!\urcorner}	%godel code
\newcommand \se [1] { \{ #1 \}}			%{x}
\newcommand\set [2]{ \{#1:#2\} }			%{x: x is gangyfan}
\newcommand\res {\!\upharpoonright\!}		%restriction
\newcommand\elem {\prec}				%elemenatary substructure
\newcommand\coll{{Coll}}
\newcommand{\card}{\operatorname{card}}
\newcommand\concat {{^\frown}}   %concatenation
\oper{\Ord}{Ord}				%ordinal

\oper{\ZFC}{ZFC}				%ZFC
\oper{\rank}{rank}				%rank
\oper{\crit}{cr}					%critical point
\oper{\crt}{crt}					%critical point
\oper{\cf}{cf}					%cofinality
\oper{\height}{ht}				%height
\oper{\wfcore}{wfcore}					%wellfounded core
\oper{\core}{core}					%core

\oper{\Ult}{Ult}				%ult
\oper{\Cone}{Cone}				%cone
\oper{\dirlim}{dirlim}
\oper{\rud}{rud}		%rud closure
\oper{\const}{const}
\oper{\OD}{OD}
\oper{\final}{final}
\oper{\HYP}{HYP}
\oper{\wfp}{wfp}
\oper{\Hom}{Hom}
\new{\ult} {\Ult}
\oper{\dom}{dom}
\oper{\rep}{rep}

\oper{\suc}{succ}
\oper{\fac}{fac}

\oper{\Code}{Code}

\oper{\ran}{ran}
\oper{\maxdom}{maxdom}
\oper{\maxran}{maxran}

\oper{\lp}{Lp} %lower part closure
\oper{\pro}{pro} %promotion
\oper{\lift}{lift} %promotion
\opera{\drop}{drop} %drop
\oper{\base}{base} %base

\newcommand\iniseg {\vartriangleleft}		%proper initial segment
\newcommand\lengthen {\triangleright}		%proper lengthening
\newcommand\inisegeq {\trianglelefteq}	%initial segment
\newcommand\isom {\cong}			%isomorphic
\newtheorem{theorem}{Theorem}[section]

\newtheorem{lemma}[theorem]{Lemma}

\newtheorem{corollary}[theorem]{Corollary}

\theoremstyle{definition}

\newtheorem{definition}[theorem]{Definition}

\theoremstyle{definition} \newtheorem*{defplain*}{\gy{Definition}}
\theoremstyle{definition} \newtheorem*{thmplain*}{\gy{Theorem}}
\theoremstyle{definition} \newtheorem*{queplain*}{\gy{Question}}
\theoremstyle{definition} 
\theoremstyle{definition} \newtheorem*{corplain*}{\gy{Corollary}}

\theoremstyle{remark} 		
\theoremstyle{remark}

\newcommand{\boldpi}[1]{{\boldsymbol{\Pi}^1_{#1}}}

\newcommand{\bolddelta}[1]{{\boldsymbol{{\delta}}^1_{#1}}}
\newcommand{\boldDelta}[1]{{\boldsymbol{{\Delta}}^1_{#1}}}

\newcommand{\DEF}{=_{{\textrm{DEF}}}}

\newcommand{\wocode}[1]{  \|   #1  \|}

\newcommand{\comm}[1]{{}}

\title{Lightface mice with finitely many Woodin cardinals from optimal determinacy hypotheses}
\date{\today{}}
\author{Yizheng Zhu}
\affil{Institut f\"{u}r mathematische Logik und Grundlagenforschung \\
Fachbereich Mathematik und Informatik\\
Universit\"{a}t M\"{u}nster\\
Einsteinstr. 62 \\
48149 M\"{u}nster, Germany 
}
\begin{document}
\maketitle{}
\begin{abstract}
  The determinacy of lightface $\Delta^1_{2n+2}$ and boldface $\boldsymbol{\Pi}^1_{2n+1}$ sets implies the existence of an $(\omega, \omega_1)$-iterable $M_{2n+1}^{\#}$. 
\end{abstract}

\section{Introduction}
\label{sec:introduction}

We prove the following theorem on the equivalence of determinacy principles and the existence of an iterable mouse with an odd number of Woodin cardinals:
\begin{theorem}
  \label{thm:main}
Suppose $n$ is a natural number.   The following are equivalent:
  \begin{enumerate}
  \item $\boldpi{2n+1}$-determinacy $+ \Delta^1_{2n+2}$-determinacy.
  \item $\forall x \in \mathbb{R}$(there is an $(\omega_1,\omega_1)$-iterable $M_{2n}^{\#}(x)$) and there is $N\in HC$ such that $L[N] \models $``there are $2n+1$ Woodin cardinals''.
  \item There is an $(\omega, \omega_1)$-iterable $M_{2n+1}^{\#}$.
  \end{enumerate}
\end{theorem}

Theorem~\ref{thm:main} solves the conjecture in \cite[Section 4.2]{SUW} for odd $n$. 
The new ingredient in this paper is the direction 2 $\Rightarrow$ 3 for $n>0$. The proof of 3 $\Rightarrow$ 1 appears in \cite{nee_opt_I,nee_opt_II}; 1 $\Rightarrow$ 2 appears in \cite{SUW}; 2 $\Rightarrow$ 3 for $n=0$ appears in \cite{hod_as_a_core_model}. 

In the proof of 2 $\Rightarrow$ 3 for $n=0$ in \cite{hod_as_a_core_model}, the key idea of producing an iterable $M_1^{\#}$ is the ``bad sequence argument'': If $(\mathcal{T}_i: i < \omega)$ is a stack of iteration trees on $\mathcal{N}$ according to an iteration strategy, $\mathcal{N}_i$ is the last model of $\mathcal{T}_i$ and $\alpha \in \mathcal{N}_i$ for any $i$, 
then for all but finitely $i$, $\pi^{\mathcal{T}_i}$ exists and $\pi^{\mathcal{T}_i}(\alpha) = \alpha$. In practice, we take $\alpha$ to be the G\"{o}del code of $(u_1,\dots,u_m)$ for any finite $m$ in order to get proper class models whose iteration strategies respect $(u_1,\dots,u_m)$, and finally by varying $m$, the pseudo-comparison of these proper class models leads to an iterable mouse with a sharp on top of a Woodin cardinal. 

This paper generalizes the ``bad sequence argument'' to the higher levels in the projective hierarchy. The main obstacle was the following: Say $n=1$. The set of reals coding countable initial segments of ${L}$ has complexity $\Pi^1_1$. However, the set of reals coding countable initial segments of $M_2$ is not $\Pi^1_3$! Due to this problem in complexity, the usual indiscernability arguments does not work any more with indiscernibles of $M_2$ \emph{above} the Woodin cardinals of $M_2$. 
Here is the correct intuition: The correct higher level analog of $L$ is not $M_2$, but rather $L[T_3]$, where $T_3$ is the Moschovakis tree on $\omega \times \bolddelta{3}$ arising from the $\Pi^1_3$-scale on the universal $\Pi^1_3$ set. ``countable initial segments of $L$'' should correspond to $\Pi^1_3$-iterable mice, as defined in \cite{steel_projective_wo_1995}. $\Pi^1_3$-iterable mice is precisely the collection of mice that are strictly smaller that $M_2|\delta^{M_2}$ in the Dodd-Jensen prewellordering of mice, where $\delta^{M_2}$ is the smallest Woodin of $M_2$. Instead of working with indiscernibles of $M_2$ \emph{above} its Woodins,  one needs to work with indiscernibles for $L[T_3]$, or essentially, indiscernibles for iterates of $M_2$ \emph{below} their Woodins. The indiscernibles for $L[T_3]$ and its relationship with $M_2^{\#}$ is worked out in \cite{sharpI}. 

We briefly recall the background knowledge. Assume $\boldpi{2}$-determinacy. Moschovakis \cite{mos_dst} shows that $\Pi^1_3$ has the scale property. $T_3$ is the tree of the $\Pi^1_3$-scale on the universal $\Pi^1_3$ set. Steel \cite{steel_projective_wo_1995} defines the notion of $\Pi^1_3$-iterable mouse. In this paper, $\Pi^1_3$-iterable mice are by default countable and 2-small. For any real $x$, the set of reals coding $\Pi^1_3$-iterable $x$-mice is $\Pi^1_3(x)$, uniformly in $x$. $\Pi^1_3$-iterable $x$-mice are genuinely $(\omega_1,\omega_1)$-iterable. If $\mathcal{M}$ and $\mathcal{N}$ are both $\Pi^1_3$-iterable $x$-mouse, $\mathcal{M} \leq_{DJ(x)} \mathcal{N}$ means that in the comparison between $\mathcal{M}$ and $\mathcal{N}$, the main branch on the $\mathcal{M}$-side does not drop. $\leq_{DJ(x)}$ is a prewellordering on the set of $\Pi^1_3$-iterable $x$-mouse.
We denote by
\begin{displaymath}
  \wocode{\mathcal{M}}_{DJ(x)}
\end{displaymath}
the $\leq_{DJ(x)}$-rank of $\mathcal{M}$. The length of $\leq _{DJ(x)}$ is at most $\bolddelta{3}$. 
If $\mathcal{M}$ is a $\Pi^1_3$-iterable $x$-mouse, the following sets are $\Delta^1_3(x)$, uniformly in $x$:
\begin{align*}
  \{ z : &~ z \text{ codes a $\Pi^1_3$-iterable $x$-mouse } \mathcal{N}_z  \wedge \wocode{\mathcal{N}_z}_{DJ(x)} < \wocode{\mathcal{M}}_{DJ(x)} \},\\
  \{ z : &~ z \text{ codes a $\Pi^1_3$-iterable $x$-mouse } \mathcal{N}_z  \wedge \wocode{\mathcal{N}_z}_{DJ(x)} = \wocode{\mathcal{M}}_{DJ(x)} \}.
\end{align*}
If $\mathcal{M}$ is a $\Pi^1_3$-iterable $x$-mouse,
\begin{displaymath}
\mathcal{M}_{\infty}^x
\end{displaymath}
denotes the direct limit of all countable non-dropping iterates of $\mathcal{M}$ and
\begin{displaymath}
\pi_{\mathcal{M}, \infty}^x: \mathcal{M} \to \mathcal{M}_{\infty}^x
\end{displaymath}
denotes the direct limit map. $\mathcal{M}_{\infty}^x$ depends only on $x$ and $\wocode{\mathcal{M}}_{DJ(x)}$, so for $\alpha = \wocode{\mathcal{M}}_{DJ(x)}$, we denote
\begin{displaymath}
  \mathcal{N}_{\alpha,\infty}^x = \mathcal{M}_{\infty}^x.
\end{displaymath}
If $A$ is a countable self-wellordered set, we can make sense of $\Pi^1_3$-iterable $A$-mice and $\leq_{DJ(A)}$, $\wocode{\cdot}_{DJ(A)}$, $\mathcal{M}^A_{\infty}$, $\pi^A_{\mathcal{M},\infty}$, $\mathcal{N}^A_{\alpha, \infty}$. As a consequence of Silver's dichotomy on $\Delta^1_3$-equivalence relations (cf.\ \cite{hjorth_coarse}, \cite[Corollary 2.14]{sharpI}) and $Q$-theory (cf.\ \cite{Q_theory,becker_kechris_1984,Kechris_Martin_II}), we are able to compare the Dodd-Jensen rank of $\Pi^1_3$-mice over different reals in a $\Sigma^1_4$ way that is absolute between transitive models closed under the $M_1^{\#}$-operator:
\begin{theorem}[{\cite[Corollary 2.15]{sharpI}}]
  \label{thm:compare_mouse_order}
  Assume $\boldDelta{2}$-determinacy. Then the relations
  \begin{align*}
&    z \text{ codes a $\Pi^1_3$-iterable $x$-mouse } \mathcal{P}_z \wedge  z' \text{ codes a $\Pi^1_3$-iterable $x'$-mouse } \mathcal{P}_{z'} \\
&\wedge \wocode{\mathcal{P}_z}_{DJ(x)} =  \wocode{\mathcal{P}_{z'}}_{DJ(x')}
  \end{align*}
and
  \begin{align*}
&    z \text{ codes a $\Pi^1_3$-iterable $x$-mouse } \mathcal{P}_z \wedge  z' \text{ codes a $\Pi^1_3$-iterable $x'$-mouse } \mathcal{P}_{z'} \\
&\wedge \wocode{\mathcal{P}_z}_{DJ(x)} =  \wocode{\mathcal{P}_{z'}}_{DJ(x')} \wedge \\
& m\in \omega \text{ codes $\alpha \in \mathcal{P}_z$ relative to }z \wedge m'\in \omega \text{ codes $\alpha' \in \mathcal{P}_{z'}$ relative to }z'\\
&  \wedge \pi^x_{\mathcal{P}_z,\infty} (\alpha) = \pi^{x'}_{\mathcal{P}_{z'},\infty} (\alpha')
  \end{align*}
are both $\Sigma^1_4$ and absolute between transitive models which contain $\se{z,x,z',x'}$ and are closed under the $M_1^{\#}$-operator. 
\end{theorem}

Assume $\forall x \in \mathbb{R} $(there is an $(\omega_1,\omega_1)$-iterable $M_2^{\#}$). Steel \cite{steel_hodlr_1995} shows that:
\begin{enumerate}
\item For any real $x$, $\leq_{DJ(x)}$ has length $\bolddelta{3}$.
\item Let $M_{2,\infty}^{\#}(x)$ be the direct limit of all countable non-dropping iterates of $M_2^{\#}(x)$. Then $\bolddelta{3}$ is the least strong up to the least Woodin in $M_{2,\infty}^{\#}(x)$ and $M_{2,\infty}^{\#}(x) = L_{\bolddelta{3}}[T_3,x]$. 
\end{enumerate}
We say that a $\Pi^1_3$-iterable $x$-mouse $\mathcal{P}$ is \emph{full} iff for any $\Pi^1_3$-iterable $\mathcal{P}$-mouse $\mathcal{R}$, $\mathcal{R}$ can be regarded as an $x$-mouse, i.e., for any $\rho < o(\mathcal{P})$, for any $A \subseteq \rho$, $A \in \mathcal{P}$ iff $A \in \mathcal{R}$. Equivalently, $\mathcal{P}$ is full iff $M_2(\mathcal{P})$ does not contain bounded subsets of $o(\mathcal{P})$ that are not in $\mathcal{P}$. If $\mathcal{P}$ is full, then $\mathcal{P}_{\infty}^x = M_{2,\infty}^{\#}(x) | \gamma$ where $\gamma = o(\mathcal{P}_{\infty}^x)$ is a cardinal and cutpoint in $M_{2,\infty}^{\#}(x)$. 

% We introduce the results concerning $L[T_3]$ and its indiscernibles as black boxes. 

% (Steel)...

Put $\mathbb{L}[T_3]= \bigcup_{x \in \mathbb{R}} \mathbb{L}[T_3]$. The following theorem shows the equivalence of $L[T_3,x]$-indiscernibles and $M_2^{\#}(x)$:

  \begin{theorem}[Zhu \cite{sharpI}]
    \label{thm:higher_sharp}
    There are countably complete $\mathbb{L}[T_3]$-measures $(\mu_n: n < \omega)$ on $(\bolddelta{3})^2$ such that for any $x \in \mathbb{R}$,
    \begin{enumerate}
    \item for $\mu_n$-a.e.\ $(\alpha,\beta)$, if $\wocode{\mathcal{R}}_{DJ(x)} = \beta$, then $\mathcal{R}$ is full and $\mathcal{R}_{\infty} = \mathcal{N}_{\beta,\infty}^x =  M_{2,\infty}^{\#}(x) | \beta$;
    \item letting
      \begin{displaymath}
        (x^{3\#})_n = \set{\gcode{\varphi}}{ \text{for $\mu_n$-a.e.\ } (\alpha,\beta), \mathcal{N}_{\beta,\infty}^x \models \varphi(\alpha)}
      \end{displaymath}
and
\begin{displaymath}
  x^{3\#} = \oplus_{n<\omega} (x^{3\#})_n,
\end{displaymath}
then
\begin{displaymath}
  x^{3\#} \equiv_m M_2^{\#}(x),
\end{displaymath}
uniformly in $x$. 
    \end{enumerate}
  \end{theorem}

Fixing $n$, 
$\mu_n$ is the higher level analog of the $\mathbb{L} \DEF \bigcup_{x \in \mathbb{R}} L[x]$-measure $\nu_n$ on $(\omega_1)^n$, where $A \in \nu_n$ iff for some $x \in \mathbb{R}$, $A$ contains all the increasing $n$-tuples of countable $x$-indiscernibles. For the reader familiar with \cite{sharpI}, $\mu_n$ can be taken as the $\mathbb{L}[T_3]$-measure arising from the level-3 tree $Y_n$ so that $\llbracket \emptyset \rrbracket_{Y_n} = u_n+\omega$. 

\section{The bad sequence argument}
\label{sec:proof}

We prove 2 $\Rightarrow$ 3 in Theorem~\ref{thm:main} for $n=1$. The general case makes no essential difference based on \cite{sharpI}.

% If $\mathcal{N}$ is a premouse, the $L[\vec{E}]$-construction based on $\mathcal{N}$ is the Mitchell-Steel array $(M_{\eta}, \mathcal{N}_\eta : \eta < \alpha)$ as defined in [] with the exception that $M_0 = \mathcal{N}$. If the construction reaches $N_{\eta} \lengtheneq \mathcal{N}$ such that $\mathcal{N}_{\eta}$ projects across $\mathcal{N}$ and all proper initial segments of $\mathcal{N}_{\eta}$ are 2-small above $\mathcal{N}$, then for the least such $\eta$, define
% \begin{displaymath}
%   M_2^{*}(\mathcal{N}) =\text{the $o(\mathcal{N})$-core of } \mathcal{N}_{\eta}. 
% \end{displaymath}
% The \emph{partial iteration strategy guided by 2-small mice certified by $L[\vec{E}]$-constructions} is the partial strategy $\Sigma$ so that $\Sigma(\mathcal{T}) = b$ iff $\mathcal{Q}(b,\mathcal{T}) =M_2^{*}(\mathcal{M}(\mathcal{T}))$ is 2-small.

% Whenever we write $M_2^{\#}(\mathcal{P})$ for some premouse $\mathcal{P}$, we mean either the least 2-small mouse extending $\mathcal{P}$ but projecting across $\mathcal{P}$ or $M_2^{\#}(\mathcal{P})$ if there is no such mouse. 

\begin{definition}
  \label{def:suitable}
  A premouse $\mathcal{P}$ is suitable iff there is $\delta \in \mathcal{P}$ such that
  \begin{enumerate}
  \item $\mathcal{P} = M_2^{\#}(\mathcal{P}|\delta) | (\delta^{+})^{M_2^{\#}(\mathcal{P}|\delta)}$.
  \item $\mathcal{P}$ satisfies the following.
    \begin{enumerate}
    \item $\delta$ is Woodin.
    \item $\forall \eta < \delta~ \forall a \in \mathcal{P}|\eta$ (the $L[\vec{E}]$-construction above $a$ with critical points above $\eta$ reaches $M_2^{\#}(a)$). 
If $\mathcal{N}$ is a $\card(\mathcal{N})^{+}+1$-iterable mouse (over $\emptyset$), then let
\begin{displaymath}
M_2^{*}(\mathcal{N}) = M_2^{\#}(\mathcal{N}) | \alpha
\end{displaymath}
regarded as a $\emptyset$-mouse, where $\alpha$ is the least such that $\exists \rho < o(\mathcal{N}) \exists A \subseteq \rho (A \in \rud(M_2^{\#}(\mathcal{N})|\alpha) \setminus \mathcal{N})$. The \emph{partial iteration strategy guided by 2-small mice} is the partial strategy $\Sigma$ so that $\Sigma(\mathcal{T}) = b$ iff $\mathcal{Q}(b,\mathcal{T}) =M_2^{*}(\mathcal{M}(\mathcal{T})) \neq M_2^{\#}(\mathcal{M}(\mathcal{T}))$.
    \item $\forall \eta < \delta ~ \mathcal{P}|\delta$ is $(\eta,\eta)$-iterable according to the partial iteration strategy guided by 2-small mice.
    \item $\forall \eta < \delta~ M_2^{\#}(\mathcal{P}|\eta) \models  \eta$ is not Woodin.  
    \end{enumerate}
  \end{enumerate}
\end{definition}
If $\mathcal{P}$ is suitable, $\delta^{\mathcal{P}}$ denotes its Woodin, and $\mathcal{P}^{-} = \mathcal{P} | \delta^{\mathcal{P}}$. If $\mathcal{P}$ is also countable, $\mathcal{P}$ itself can be regarded as a full $\Pi^1_3$-iterable $\mathcal{P}^{-}$-mouse. In fact, a countable premouse $\mathcal{P}$ is suitable iff $\mathcal{P}$ satisfies the first order property in Clause 2 in Definition~\ref{def:suitable} and $\mathcal{P}$ is full. 
If $\widehat{\mathcal{P}}$ is another $\Pi^1_3$-iterable $\mathcal{P}^{-}$-mouse, $\widehat{\mathcal{P}}$ can also be regarded as a $\emptyset$-premouse, and we have $\mathcal{P} \inisegeq \widehat{\mathcal{P}}$ iff $\widehat{\mathcal{P}}$ is full. 

\begin{theorem}[Mitchell-Steel \cite{fsit}]
  \label{thm:MS}
  If $\forall x \in \mathbb{R}$(there is an $(\omega_1,\omega_1)$-iterable $M_2^{\#}(x)$) and $\exists N \in HC~ L[N] \models $``there are three Woodin cardinals'', then there is a countable suitable premouse. 
\end{theorem}

% In the rest of this paper, every premouse is countable. 
The following condensation principle is an easy generalization of \cite[Lemma 3.3]{hod_as_a_core_model}. Its proof can be found in e.g.\ \cite[Lemma 3.3]{SUW}. 
  \begin{lemma}
    \label{lem:condensation_suitable}
    If $\mathcal{P}$ is countable and suitable, $\widehat{\mathcal{P}}$ is a $\Pi^1_3$-iterable $\mathcal{P}^{-}$-mouse, $\mathcal{H}$ is the transitive collapse of $Hull_{\omega}^{\widehat{\mathcal{P}}}$, the $\widehat{\mathcal{P}}$-definable points where $\widehat{\mathcal{P}}$ is regarded as a $\emptyset$-premouse, 
% $\mathcal{N}\elem \widehat{\mathcal{P}}$, $\mathcal{N} \in \widehat{\mathcal{P}}$ and countable in $\widehat{\mathcal{P}}$, $\mathcal{H}$ is the transitive collapse of $\mathcal{N}$, 
then $\mathcal{H}$ (regarded as a $\emptyset$-premouse) is an initial segment of $\mathcal{P}|\omega_1^{\mathcal{P}}$.  
  \end{lemma}

  \begin{definition}
    \label{def:short_tree}
    Let $\mathcal{T}$ be a normal iteration tree on a suitable $\mathcal{P}$. $\mathcal{T}$ is \emph{short} iff $\forall \alpha \leq \lh(\mathcal{T})$ limit, $M_2(\mathcal{M}(\mathcal{T}\res \alpha)) \models ``\delta(\mathcal{T}\res \alpha)$ is not Woodin''. $\mathcal{T}$ is \emph{maximal} iff $\mathcal{T}$ is not short. 
  \end{definition}

  \begin{definition}
    \label{def:short_tree_iterable}
    Suppose $\mathcal{P}$ is suitable. $\mathcal{P}$ is \emph{short-tree-iterable} iff for any putative short tree $\mathcal{T}$ on $\mathcal{P}$,  for any $\Pi^1_3$-iterable $\mathcal{P}^{-}$-mouse $\widehat{\mathcal{P}}$, letting $\widehat{\mathcal{T}}$ be $\mathcal{T}$ construed as a putative tree on $\widehat{\mathcal{P}}$,
    \begin{enumerate}
    \item if $\widehat{\mathcal{T}}$ has a last model $\mathcal{M}_{\alpha}^{\widehat{\mathcal{T}}}$, then either
      \begin{enumerate}
      \item $[0,\alpha]_T$ drops, $\mathcal{M}_{\alpha}^{{\mathcal{T}}}$  is $\Pi^1_3$-iterable, or
      \item $[0,\alpha]_T$ does not drop, $\mathcal{M}_{\alpha}^{\widehat{\mathcal{T}}}$ is a $\Pi^1_3$-iterable $\pi_{0,\alpha}^{\mathcal{T}}(\mathcal{P}^{-})$-mouse. 
      \end{enumerate}
    \item If $\lh(\mathcal{T})$ is limit, $\mathcal{T}$ is short, then $\mathcal{T}$ has a cofinal branch $b$ such that $\mathcal{Q}(b,\mathcal{T}) = M_2^{*}(\mathcal{M}(\mathcal{T}))$, where $M_2^{*}(\mathcal{M}(\mathcal{T})) = M_2^{\#}(\mathcal{M}(\mathcal{T}))|\alpha$ regarded as a $\emptyset$-mouse, $\alpha$ is the least such that  $\exists \rho < \delta(\mathcal{T}) \exists A \subseteq \rho$ $ (A \in M_2^{\#}(\mathcal{M}(\mathcal{T}))|\alpha+1 \setminus \mathcal{M}(\mathcal{T}))$.  
    \end{enumerate}
  \end{definition}

  \begin{lemma}
    \label{lem:suitable_is_short_tree_iterable}
    If $\mathcal{P}$ is suitable, then $\mathcal{P}$ is short-tree-iterable. 
  \end{lemma}
  \begin{proof}
    Suppose not. There is then a putative short tree $\mathcal{T}$ on $\mathcal{P}$ and a  $\Pi^1_3$-iterable $\mathcal{P}^{-}$-mouse $\widehat{\mathcal{P}}$ such that either
    \begin{enumerate}
    \item $\lh(\mathcal{T}) = \alpha+1$ is a successor, $[0,\alpha]_T$ drops,  $\mathcal{M}_{\alpha}^{\mathcal{T}}$ is not $\Pi^1_3$-iterable, or 
    \item $\lh(\mathcal{T}) = \alpha+1$ is a successor, $[0,\alpha]_T$ does not drop, letting $\widehat{\mathcal{T}}$  be $\mathcal{T}$ construed as a putative tree on $\mathcal{P}$, then $\mathcal{M}_{\alpha}^{\widehat{\mathcal{T}}}$ is not a $\Pi^1_3$-iterable $\pi_{0,\alpha}^{\mathcal{T}}(\mathcal{P}^{-})$-mouse, or
    \item $\lh(\mathcal{T}) = \lambda$ is a limit, there is a $\delta(\mathcal{T})$-sound, $\Pi^1_3$-iterable $\mathcal{M}(\mathcal{T})$-mouse $\mathcal{R}$ that can be regarded as an $\emptyset$-premouse with $\rho_{\omega}(\mathcal{R}) < \delta(\mathcal{T})$, but there is no cofinal branch $b$ such that $\mathcal{Q}(b,\mathcal{T}) =\mathcal{R}$.
% , where $\mathcal{R}$ is the least $\Pi^1_3$-iterable $\mathcal{M}(\mathcal{T})$-mouse projecting across $\delta(\mathcal{T})$.
    \end{enumerate}
The existence of a $\mathcal{P}$-bad pair $(\mathcal{T}, \widehat{\mathcal{P}})$ is $\Sigma^1_4$ in the code of $\mathcal{P}$. By Steel \cite{steel_projective_wo_1995}, $M_2(z) \elem_{\Sigma^1_4} V$ for any real $z$. Hence, a bad pair can be found in $M_2(\mathcal{P})^{\coll(\omega,\mathcal{P})}$. Working in $M_2(\mathcal{P})$, take a countable elementary substructure $\mathcal{N} \elem M_2(\mathcal{P}) | \eta$, where $\eta$ is the successor of $o(\mathcal{P})$ in $M_2(\mathcal{P})$. $\mathcal{H}$ is the transitive collapse of $\mathcal{N}$, which is by Lemma~\ref{lem:condensation_suitable} an initial segment of $\mathcal{P}$. Let $\mathcal{Q}$ be the image of $\mathcal{P}$ under the transitive collapsing map. Take $g \in \mathcal{P}$ which is generic over $\mathcal{H}$ for $\coll(\omega, \mathcal{Q})$.  So $\mathcal{H}[g] \models ``$there is a $\mathcal{Q}$-bad pair $(\mathcal{U}, \widehat{\mathcal{Q}})$''. Note that $\mathcal{H}[g] \models ``$I am closed under the $M_1^{\#}$-operator'', therefore as $\mathcal{H} \inisegeq \mathcal{P}$, the $M_1^{\#}$-operators are computed correctly in $\mathcal{H}[g]$, which implies that $\mathcal{H}[g] \elem_{\Sigma^1_3} \mathcal{P}$ by genericity iterations (cf.\ \cite[Section 7.2]{steel-handbook}). 
So  $(\mathcal{U}, \widehat{\mathcal{Q}})$, being a $\mathcal{Q}$-bad pair from the point of view of $\mathcal{H}[g]$, is also seen as a $\mathcal{Q}$-bad pair in $\mathcal{P}$. However, $\mathcal{Q} \iniseg \mathcal{P}$ and $\mathcal{Q}$ is $(\omega_1,\omega_1)$-iterable in $\mathcal{P}$ by suitability. Contradiction!
  \end{proof}

If $\mathcal{P}$ is suitable and $\mathcal{T}$ is a short tree on $\mathcal{P}$ such that $\pi^{\mathcal{T}}$ exists, we can define an order preserving function
\begin{displaymath}
  g^{\mathcal{T}} : \bolddelta{3} \to \bolddelta{3}
\end{displaymath}
as follows: If $\mathcal{P}'$ is a $\Pi^1_3$-iterable $\mathcal{P}^{-}$-mouse, let $f^{\mathcal{T}}(\mathcal{P}')$ be the last model of $\mathcal{T}$ construed as a tree on $\mathcal{P}'$, and define
\begin{displaymath}
  g^{\mathcal{T}} (\wocode{\mathcal{P}'}_{\mathcal{P}^{-}}) = \wocode{ f^{\mathcal{T}} (\mathcal{P}') }_{\pi^{\mathcal{T}}(\mathcal{P}^-)}.
\end{displaymath}
$g^{\mathcal{T}}$ is well-defined: Suppose $\wocode{\mathcal{P}'}_{\mathcal{P}^{-}} = \wocode{\mathcal{P}''}_{\mathcal{P}^{-}}$ and suppose without loss of generality that $\mathcal{P}''$ is a nondropping iterate of $\mathcal{P}'$ via $\mathcal{U}$ above $\mathcal{P}^{-}$. We would like to show that $\wocode{f^{\mathcal{T}}(\mathcal{P}')}_{\pi^{\mathcal{T}}(\mathcal{P}^-)} = \wocode{f^{\mathcal{T}}(\mathcal{P}'')}_{\pi^{\mathcal{T}}(\mathcal{P}^-)} $. On the one hand, the tree $\mathcal{P}'$-to-$f(\mathcal{P}')$ is copied to the tree $\mathcal{P}''$-to-$f(\mathcal{P}'')$ according to $\pi^{\mathcal{U}}$ (both trees are just $\mathcal{T}$ construed on different models), so $\pi^{\mathcal{U}}$ induces a copying map from $f^{\mathcal{T}}(\mathcal{P}')$ to $f^{\mathcal{T}}(\mathcal{P}'')$, giving that  $\wocode{f^{\mathcal{T}}(\mathcal{P}')}_{\pi^{\mathcal{T}}(\mathcal{P}^-)} \leq \wocode{f^{\mathcal{T}}(\mathcal{P}'')}_{\pi^{\mathcal{T}}(\mathcal{P}^-)} $. 
On the other hand, we can copy $\mathcal{U}$ to a tree on $f(\mathcal{P}')$ according to the iteration map from $\mathcal{P}'$ to $f(\mathcal{P}'')$, leading to an iteration tree $\mathcal{V}$ on $f(\mathcal{P}')$ with last model $\mathcal{Q}$ so that $\pi^{\mathcal{V}}$ exists. Note that $\mathcal{U}$ is above $\mathcal{P}^{-}$ while $\mathcal{T}$ is based on $\mathcal{P}^{-}$. The technique in \cite[Lemma 3.2]{MR3397344} enables us to define a map from $f(\mathcal{P}'')$ to $\mathcal{Q}$, giving that  $\wocode{f^{\mathcal{T}}(\mathcal{P}'')}_{\pi^{\mathcal{T}}(\mathcal{P}^-)} \leq  \wocode{f^{\mathcal{T}}(\mathcal{P}')}_{\pi^{\mathcal{T}}(\mathcal{P}^-)} $. A similar argument shows that $g^{\mathcal{T}}$ is order preserving. 

  \begin{corollary}
    \label{coro:suitable_short_tree_final_is_suitable}
    Suppose $\mathcal{P}$ is suitable and $\mathcal{T}$ is a short tree on $\mathcal{P}$ with last model $\mathcal{Q}$ such that $\pi^{\mathcal{T}}$ exists. Then $\mathcal{Q}$ is suitable. 
  \end{corollary}
  \begin{proof}
All the first-order-in-$\mathcal{P}$ properties in Definition~\ref{def:suitable} are preserved by elementarity. We need to show fullness.
 For any ${\mathcal{P}}'$, a full $\Pi^1_3$-iterable $\mathcal{P}^{-}$-mouse, ${\mathcal{P}}'\models o(\mathcal{P}) = (\delta^{\mathcal{P}})^{+}$, and hence by elementarity, $f^{\mathcal{T}}(\mathcal{P}') \models o(\mathcal{Q})  =  (\pi^{\mathcal{T}}(\delta^{\mathcal{P}}))^{+}$. If $\mathcal{Q}'$ is any full $\Pi^1_3$-iterable $\mathcal{Q}|\pi^{\mathcal{T}}(\delta^{\mathcal{P}})$-mouse, we may pick such $\mathcal{P}'$ with $g^{\mathcal{T}} (\wocode{\mathcal{P}'}_{\mathcal{P}^{-}}) > \wocode{\mathcal{Q}'}_{\pi^{\mathcal{T}}(\mathcal{P}^{-})}$, implying that  $\mathcal{Q}' \models o(\mathcal{Q})  =  (\pi^{\mathcal{T}}(\delta^{\mathcal{P}}))^{+}$. Hence $\mathcal{Q}$ is suitable and $\delta^{\mathcal{Q}} = \pi^{\mathcal{T}}(\delta^{\mathcal{P}})$. 
  \end{proof}

  \begin{definition}
    \label{def:pseudo_normal_iterate}
    Let $\mathcal{P}$ be suitable. $\mathcal{Q}$ is called a \emph{pseudo-normal-iterate} of $\mathcal{P}$ iff $\mathcal{Q}$ is suitable and there is a normal tree $\mathcal{T}$ on $\mathcal{P}$ such that either $\mathcal{Q}$ is the last model of $\mathcal{T}$, $\pi^{\mathcal{T}}$ exists, or $\mathcal{Q} = M_2(\mathcal{M}(\mathcal{T})) | ((\delta(\mathcal{T}))^{+})^{M_2(\mathcal{M}(\mathcal{T}))}$.
  \end{definition}

  \begin{definition}
    \label{def:finite_full_stack}
    Let $\mathcal{P}$ be suitable. $((\mathcal{T}_i: i < k), (\mathcal{P}_i: i \leq k))$ is called a \emph{finite full stack} on $\mathcal{P}$ iff $\mathcal{P}_0 = \mathcal{P}$ and for each $i$, $\mathcal{P}_{i+1}$ is a pseudo-normal-iterate of $\mathcal{P}_i$ witnessed by $\mathcal{T}_i$. 
  \end{definition}

% For a countable self-wellordered set $X$, define $\mathcal{N}_{\alpha,\infty}(X) = {M}_{\infty} ( \mathcal{R} )$,  where $\mathcal{R}$ is a $\Pi^1_3$-iterable $X$-mouse of $<_{DJ(\mathcal{X})}$-rank $\alpha$.  

  \begin{definition}
    \label{def:th_s}
    Suppose $\mathcal{P}$ is countable and suitable, $\alpha < \beta < \bolddelta{3}$, and $\alpha < \mathcal{N}_{\beta,\infty}^{\mathcal{P}^{-}}$. Then
    \begin{enumerate}
    \item $Th^{\mathcal{P}}_{(\alpha,\beta)} = \set{(\gcode{\varphi}, \xi)}{\xi < \delta^{\mathcal{P}}, \mathcal{N}_{\beta,\infty}^{\mathcal{P}^{-}} \models \varphi(\xi, \alpha)}$.
    \item $\gamma_{(\alpha,\beta)}^{\mathcal{P}} = \sup (Hull^{\mathcal{N}_{\beta,\infty}^{\mathcal{P}^{-}}} (\se{\alpha}) \cap \delta^{\mathcal{P}})$.
    \item $H_{(\alpha,\beta)}^{\mathcal{P}}= Hull^{\mathcal{N}_{\beta,\infty}^{\mathcal{P}^{-}}} (\gamma^{\mathcal{P}}_{(\alpha,\beta)} \cup \se{\alpha})$.
    \item By Theorem~\ref{thm:higher_sharp}, define $(\mathcal{P}^{-})^{3\#}_n = Th^{\mathcal{P}}_{(\alpha,\beta)} $ for $\mu_n$-a.e.\  $(\alpha,\beta)$ and $  (\mathcal{P}^{-})^{3\#} = \oplus_{n<\omega} (\mathcal{P}^{-})^{3\#}_n.$
    \item $\gamma^{\mathcal{P}}_{3\#,n} = \gamma^{\mathcal{P}}_{(\alpha,\beta)}$ for $\mu_n$-a.e.\ $(\alpha,\beta)$.
    \item $H^{\mathcal{P}}_{3\#,n} = H^{\mathcal{P}}_{(\alpha,\beta)}$ for $\mu_n$-a.e.\ $(\alpha,\beta)$.
    \end{enumerate}
Suppose $\mathcal{T}$ is a normal iteration tree on $\mathcal{P}$ and $b$ is a cofinal branch of $\mathcal{T}$. $b$ is said to \emph{respect} $(\alpha,\beta)$ iff $\mathcal{Q}=\mathcal{M}_b^{\mathcal{T}}$ is suitable and
\begin{displaymath}
  \pi_b^{\mathcal{T}} (Th^{\mathcal{P}}_{(\alpha,\beta)}) = Th^{\mathcal{Q}}_{(\alpha,\beta)}. 
\end{displaymath}
$b$ is said to \emph{respect} $(\cdot)^{3\#}_n$ iff $\mathcal{Q} = \mathcal{M}^{\mathbb{T}}_b$ is suitable and
\begin{displaymath}
  \pi_b^{\mathcal{T}} ( (\mathcal{P}^{-})^{3\#}_n) = (\mathcal{Q} ^{-})^{3\#}_n.
\end{displaymath}
 $\mathcal{P}$ is \emph{$n$-iterable} iff for any finite full stack $((\mathcal{T}_i: i < k), (\mathcal{P}_i : i \leq k))$ on $\mathcal{P}$, there is $(b_i : i < k)$ such that each $b_i$ respects $(\cdot)^{3\#}_n$.
  \end{definition}

By Theorem~\ref{thm:higher_sharp}, for any countable suitable $\mathcal{P}$, we must have
\begin{displaymath}
  \sup_{n<\omega} \gamma^{\mathcal{P}}_{3\#,n} = \delta^{\mathcal{P}}.
\end{displaymath}

% If $\mathcal{P}$ is countable and suitable, by Theorem~\ref{thm:higher_sharp}, we may define
% \begin{displaymath}
%   (\mathcal{P}^{-})^{3\#}_n = Th^{\mathcal{P}}_{(\alpha,\beta)} \text{ for $\mu_n$-a.e.\ } (\alpha,\beta).
% \end{displaymath}
% and
% \begin{displaymath}
%   (\mathcal{P}^{-})^{3\#}_n = \oplus_{n<\omega} (\mathcal{P}^{-})^{3\#}_n.
% \end{displaymath}

  % \begin{definition}
  %   \label{def:s_iterable}
  %   Let $\alpha < \beta < \bolddelta{3}$. $\mathcal{P}$ is $(\alpha,\beta)$-iterable iff for any finite full stack $((\mathcal{T}_i: i < k), (\mathcal{P}_i: i \leq k))$ on $\mathcal{P}$, there are  $(b_i: i < k)$ such that each $b_i$ is a cofinal branch of $\mathcal{T}_i$ respecting $(\alpha,\beta)$. 
  % \end{definition}

% \begin{definition}
%   \label{def:iterable_3sharp_n}
% Let $\mathcal{P}$ be countable and suitable.  
%   \begin{enumerate}
%   \item Let $\mathcal{T}$ be a normal iteration tree on $\mathcal{P}$ and suppose $b$ is a cofinal branch of $\mathcal{T}$. 
% $b$ \emph{respects} $(\cdot)^{3\#}_n$ iff $\pi_b^{\mathcal{T}} ( (\mathcal{P}^{-})^{3\#}_n) = (\mathcal{Q} ^{-})^{3\#}_n$. 
%   \item A suitable $\mathcal{P}$ is \emph{$n$-iterable} iff for any finite full stack $((\mathcal{T}_i: i < k), (\mathcal{P}_i : i \leq k))$ on $\mathcal{P}$, there is $(b_i : i < k)$ such that each $b_i$ respects $(\cdot)^{3\#}_n$.
%   \end{enumerate}
% \end{definition}

\begin{lemma}
  \label{lem:bad_sequence}
  Let $n<\omega$. Then there is a countable, $n$-iterable suitable mouse. 
\end{lemma}
\begin{proof}
  Otherwise, there is $((\mathcal{T}_i: i < \omega), (\mathcal{P}_i: i <\omega))$ such that $\mathcal{P}_0 =\mathcal{P}$ is suitable, $\mathcal{T}_i$ is a normal tree on $\mathcal{P}_i$, but for infinitely many $i$, there is no cofinal branch $b_i$ through $\mathcal{T}_i$ that respects $(\cdot)^{3\#}_n$. 

Fix $z \in \mathbb{R}$ coding $(\vec{\mathcal{T}}, \vec{\mathcal{P}})$. 
Fix $\alpha < \beta < \bolddelta{3}$ such that for any $i$, $(\mathcal{P}_i)^{3\#}_n = Th^{\mathcal{P}}_{(\alpha,\beta)}$ and if $\wocode{\mathcal{R}}_{DJ(\mathcal{P}_i^{-})} = \beta$, then $\mathcal{R}$ is full and $\mathcal{R}_{\infty} = \mathcal{N}_{\beta,\infty} ^{\mathcal{P}_i^{-}} = M_{2,\infty}^{\#}(\mathcal{P}_i^{-}) | \beta$. Thus, for infinitely many $i$, there is no cofinal branch $b_i$ through $\mathcal{T}_i$ that respects $(\alpha, \beta)$. We call $(\vec{\mathcal{T}}, \vec{\mathcal{P}})$ an $(\alpha, \beta)$-bad sequence based on $\mathcal{P}$. 

Let $\widehat{\mathcal{P}}$ be a $\Pi^1_3$-iterable $\mathcal{P}^{-}$-mouse and $\eta \in \widehat{\mathcal{P}}$ so that $\wocode{\widehat{\mathcal{P}}}_{\mathcal{P}^{-}} = \beta$ and $\pi^{\mathcal{P}^{-}}_{\widehat{\mathcal{P}},\infty}(\eta) = \alpha$. By Theorem~\ref{thm:compare_mouse_order}, the statement
\begin{quote}
``There is an $(\pi^{\mathcal{P}^{-}}_{\widehat{\mathcal{P}},\infty}(\eta) ,\wocode{\widehat{\mathcal{P}}}_{\mathcal{P}^{-}} )$-bad sequence $(\vec{\mathcal{T}}, \vec{\mathcal{P}})$ based on $\mathcal{P}$''
\end{quote}
is $\Sigma^1_4$ in the code of $\widehat{\mathcal{P}}$ and absolute between transitive models closed under the $M_1^{\#}$-operator.  It is a true statement in $V$, so by absoluteness, true  $M_2(\widehat{\mathcal{P}})^{\coll(\omega,\widehat{\mathcal{P}})}$ as well. By our choice of $\beta$,  $\widehat{\mathcal{P}}$ is full, so $M_2(\widehat{\mathcal{P}})$ can be regarded as an $\emptyset$-premouse and $o(\widehat{\mathcal{P}})$ is a cardinal and cutpoint in $M_2(\widehat{\mathcal{P}})$.
As in the proof of Lemma~\ref{lem:suitable_is_short_tree_iterable}, we get $\mathcal{H} \iniseg \mathcal{P}$ and $g \in \mathcal{P}$ generic over $\mathcal{H}$, $ \se{\mathcal{Q}, \vec{\mathcal{Q}}, \vec{\mathcal{U}} , \xi} \in \mathcal{H}[g]$ so that
\begin{quote}
$    \mathcal{H}[g] \models $``$(\vec{\mathcal{U}}, \vec{\mathcal{Q}}) $ is an $(\pi^{\mathcal{Q}^{-}}_{\widehat{\mathcal{Q}}, \infty} (\eta) , \wocode{\widehat{\mathcal{Q}}}_{DJ(\mathcal{Q}^{-})} )$-bad sequence based on $\mathcal{Q}$.''
\end{quote}
As $\mathcal{H}[g] \elem_{\Sigma^1_3} \mathcal{P}$, we have $(\bar{\alpha},\bar{\beta}) \in \mathcal{P}$ so that  
\begin{quote}
$   \mathcal{P} \models $``$(\vec{\mathcal{U}}, \vec{\mathcal{Q}}) $ is an $(\bar{\alpha}, \bar{\beta})$-bad sequence and $\wocode{\widehat{\mathcal{Q}}}_{DJ(\mathcal{Q}^{-})} = \bar{\beta}$, $\pi^{\mathcal{Q}^{-}}_{\widehat{\mathcal{Q}}, \infty} (\eta) = \bar{\alpha}$''.
\end{quote}
% Put $\mathcal{Q}_i = \mathcal{M}(\mathcal{U}_i)$. 
For the rest of this proof, we work in $\mathcal{P}$. Pick $\widehat{\mathcal{Q}}_i $ and $\xi_i \in \widehat{\mathcal{Q}}_i$ so that
\begin{displaymath}
\mathcal{P}\models \wocode{\widehat{\mathcal{Q}}_i}_{\mathcal{Q}_i^{-}} = \bar{\beta} \wedge \pi^{{\mathcal{Q}}_i^{-}}_{\widehat{\mathcal{Q}}_i, \infty} (\xi_i) = \bar{\alpha}  .
\end{displaymath}
We define $(\mathcal{R}_i , \mathcal{S}_i, b_i, \widehat{\mathcal{U}}_i : i < \omega)$ and $(\mathcal{V}_i, \mathcal{W}_i : 1 \leq i < \omega)$ inductively such that:
\begin{enumerate}
\item $\mathcal{R}_i \lengthen \mathcal{Q}_i$, $\mathcal{R}_i$ is $\Pi^1_3$-iterable above $\mathcal{Q}_i^{-}$, $\mathcal{R}_0  = \widehat{\mathcal{Q}}_0$;
\item $\widehat{\mathcal{U}}_i $ is $\mathcal{U}_i$ construed as an iteration tree on $\mathcal{R}_i$;
\item $b_i$ is the cofinal branch of $\widehat{\mathcal{U}}_i$ chosen by the internal strategy of $\mathcal{P}$;
\item $\mathcal{S}_{i+1}$ is the last model of $\widehat{\mathcal{U}}_i \concat b_i$;
\item $(\mathcal{V}_i, \mathcal{W}_i)$ is the comparison of $(\mathcal{S}_i, \widehat{\mathcal{Q}}_i)$ and $\mathcal{R}_i$ is the last model of $\mathcal{W}_i$. 
\end{enumerate}

\begin{tikzpicture}
  \matrix (m) [matrix of math nodes,row sep=1.5em,column sep=4em,minimum width=2em]
  {
     \mathcal{R}_0  = \widehat{\mathcal{Q}_0} &  \mathcal{S}_1& & & \\
    \mathcal{Q}_0 & \mathcal{Q}_1  & & &\\
   & \widehat{\mathcal{Q}}_1 & \mathcal{R}_1 & \mathcal{S}_2  & \\
   & & \mathcal{Q}_1 & \mathcal{Q}_2 &\\
  &&&                \widehat{\mathcal{Q}_2} & \mathcal{R}_3 \\} ;
  \path[-]
    (m-1-1) edge node [above] {$\widehat{\mathcal{U}}_0, b_0$} (m-1-2)
    (m-2-1) edge node [above] {${\mathcal{U}}_0$} (m-2-2)
    (m-1-2) edge node [above] {$\mathcal{V}_1$} (m-3-3)
    (m-3-2) edge node [above] {$\mathcal{W}_1$} (m-3-3)
    (m-3-3) edge node [above] {$\widehat{\mathcal{U}}_1, b_1$} (m-3-4)
    (m-4-3) edge node [above] {${\mathcal{U}}_1$} (m-4-4)
    (m-3-4) edge node [above] {$\mathcal{V}_2$} (m-5-5)
    (m-5-4) edge node [above] {$\mathcal{W}_2$} (m-5-5);
  \path[color=white]
    (m-1-1) edge node [sloped,color=black]{$\lengthen$} (m-2-1)
(m-1-2) edge node [sloped,color=black]{$\lengthen$} (m-2-2)
(m-3-2) edge node [sloped,color=black]{$\lengthen$} (m-2-2)
(m-3-3) edge node [sloped,color=black]{$\lengthen$} (m-4-3)
(m-3-4) edge node [sloped,color=black]{$\lengthen$} (m-4-4)
(m-5-4) edge node [sloped,color=black]{$\lengthen$} (m-4-4)
;
\end{tikzpicture}

By monotonicity of the function $g^{\widehat{\mathit{U}} \concat b_i}: \bolddelta{3} \to \bolddelta{3}$, we can inductively see that for each $i$, $\wocode{\mathcal{S}_i}_{DJ(\mathcal{Q}_i^{-})} \geq \bar{\beta}$,  the main branch of $\mathcal{W}_i$ does not drop, and $\wocode{\mathcal{R}_i}_{DJ(\mathcal{Q}_i^{-})} \geq \bar{\beta}$. The stack
\begin{displaymath}
\widehat{\mathcal{U}}_0 \concat b_0 \concat \mathcal{V}_1 \concat \widehat{\mathcal{U}}_1 \concat b_1 \concat \mathcal{V}_2 \concat \dots
 \end{displaymath}
is according to the internal strategy of $\mathcal{P}$. So for some $m< \omega$, we have for any $i>m$, $\pi^{\mathcal{U}_i}_{b_i}$ exists and $\pi^{\mathcal{V}_i}$ exists. The map $\pi^{\widehat{\mathcal{U}}_i \concat b_i \concat \mathcal{V}_{i+1}}$ induces a map $\tau_i : (\mathcal{R}_i)^{\mathcal{Q}_i^{-}}_{\infty} \to (\mathcal{R}_{i+1})^{\mathcal{Q}_{i+1}^{-}}_{\infty}$ so that $\tau_i \circ \pi^{\mathcal{Q}_i^{-}}_{\mathcal{R}_i,\infty} =\pi^{\mathcal{Q}_{i+1}^{-}}_{\mathcal{R}_{i+1},\infty}  \circ  \pi^{\widehat{\mathcal{U}}_i \concat b_i \concat \mathcal{V}_{i+1}} $. Clearly $\tau_i (\bar{\beta}) \geq \bar{\beta}$. So $\pi^{\widehat{\mathcal{U}}_i \concat b_i \concat \mathcal{V}_{i+1}} \circ \pi^{\mathcal{U}_{i}} (\xi_{i}) \geq \pi^{\mathcal{U}_{i+1}} (\xi_{i+1})$. So we must have some $m < m' < \omega$ so that for any $i > m'$,  $\pi^{\widehat{\mathcal{U}}_i \concat b_i \concat \mathcal{V}_{i+1}} \circ \pi^{\mathcal{U}_{i}} (\xi_{i}) = \pi^{\mathcal{U}_{i+1}} (\xi_{i+1})$. In other words, for any $i > m'$, $b_i$ respects $(\bar{\alpha},\bar{\beta})$, contradicting to the assumption that $(\vec{\mathcal{U}}, \vec{\mathcal{Q}}) $ is an $(\bar{\alpha}, \bar{\beta})$-bad sequence.
\end{proof}

By Lemma~\ref{lem:bad_sequence}, we can find $(\mathcal{P}_n:n < \omega)$ where $\mathcal{P}_n$ is a countable, $n$-iterable suitable premouse. The pseudo-comparison leads to countable iteration trees $(\mathcal{T}_n: n < \omega)$ and a suitable $\mathcal{Q}$ so that $\mathcal{T}_n$ is an iteration tree on $\mathcal{P}_n$ with last model $\mathcal{Q}$. $\mathcal{Q}$ is then $n$-iterable for any $n<\omega$. The usual limit branching argument (cf.\ \cite[Lemma 4.12]{hod_as_a_core_model}) gives an  $(\omega,\omega_1)$-iteration strategy for $\mathcal{Q}$: For instance, suppose $\mathcal{T}$ is a normal tree on $\mathcal{Q}^{-}$ with pseudo-normal-iterate $\mathcal{R}$. Let
\begin{displaymath}
  b_i = \cap \set{b}{ b \text{ is a branch through $\mathcal{T}$} \wedge b \text{ respects } (\cdot)^{3\#}_i}.
\end{displaymath}
Then $b_i \subseteq b_{i+1}$, $\gamma^{\mathcal{M}^{\mathcal{T}}_{\max b_i}}_{3\#,i} = \gamma^{\mathcal{R}}_{3\#,i}$, and we have an isomorphism $\sigma_i:H^{\mathcal{M}^{\mathcal{T}}_{\max b_i}}_{3\#,i} \isom H^{\mathcal{R}}_{3\#,i}$. 
Let $b = \cup _{i < \omega} b_i$. Then $\delta^{\mathcal{M}^{\mathcal{T}}_{\sup b}} \geq \sup_{n<\omega} \gamma^{\mathcal{R}}_{3\#,n} = \delta^{\mathcal{R}}$. So $b$ must be a cofinal branch. There is a canonical map
\begin{displaymath}
  \tau : \mathcal{R} \to  \mathcal{M}^{\mathcal{T}}_b 
\end{displaymath}
defined by $\tau (a) =  \pi_{\alpha,b} \circ \sigma_i^{-1} (a)$ for $\alpha < \max(b_i)$ and $a \in H^{\mathcal{R}}_{3\#,i}$. $\tau$ is onto $\mathcal{M}_b^{\mathcal{T}}$ because $\mathcal{M}^{\mathcal{T}}_{\max b_i} = \cup_{n<\omega} H_{3\#,n}^{\mathcal{M}^{\mathcal{T}}_{\max b_i}}$.  Therefore, $\tau$ is the identity, $\mathcal{M}_b^{\mathcal{T}} = \mathcal{R}$ and $b$ respects $(\cdot)^{3\#}_n$. 

% If $E$ is an extender used along $b$ after $\max(b_i)$, we must have $\crt(E) > \gamma^{\mathcal{R}}_{3\#,n} = \gamma^{\mathcal{M}^T_{\max b_i}}_{3\#,n}$. 
 % $\mathcal{M}_b^{\mathcal{T}} = \mathcal{R}$ and $b$ respects $(\cdot)^{3\#}_n$. 

In other words, by Theorem~\ref{thm:higher_sharp}, $M_2^{\#}(\mathcal{Q})$, regarded as a $\emptyset$-mouse, has a partial $(\omega,\omega_1)$-iteration strategy $\Gamma$ with respect to stacks of normal trees based on $\mathcal{Q}$ that moves the top $M_2^{\#}$-component correctly, i.e., whenever $\mathcal{U}$ is according to $\Gamma$ based on $\mathcal{Q}$ and the main branch of $\mathcal{U}$ does not drop, the last model of $\mathcal{U}$ must be $M_2^{\#}(\pi^{\mathcal{U}}(\mathcal{Q}))$. Also by definition of suitability, whenever $\mathcal{U}$ is according to $\Gamma$ based on $\mathcal{Q}^{-}$ but the main branch of $\mathcal{U}$ drops, the last model of $\mathcal{U}$ is $\Pi^1_3$-iterable. 
By the technique in \cite{MR2609944}, $M_2^{\#}(\mathcal{Q})$ is $(\omega,\omega_1)$-iterable. $M_2^{\#}(\mathcal{Q})$ has a sharp above three Woodins, so by taking its $\Sigma_1$-Skolem hull, we get an $(\omega,\omega_1)$-iterable $M_3^{\#}$. This finishes the proof of Theorem~\ref{thm:main} for $n=1$. 

\bibliographystyle{hplain}
\bibliography{m3}

\begin{thebibliography}{10}

\bibitem{becker_kechris_1984}
Howard~S. Becker and Alexander~S. Kechris.
\newblock Sets of ordinals constructible from trees and the third {V}ictoria
  {D}elfino problem.
\newblock In {\em Axiomatic set theory ({B}oulder, {C}olo., 1983)}, volume~31
  of {\em Contemp. Math.}, pages 13--29. Amer. Math. Soc., Providence, RI,
  1984.

\bibitem{MR2609944}
Gunter Fuchs, Itay Neeman, and Ralf Schindler.
\newblock A criterion for coarse iterability.
\newblock {\em Arch. Math. Logic}, 49(4):447--467, 2010.

\bibitem{hjorth_coarse}
Greg Hjorth.
\newblock Some applications of coarse inner model theory.
\newblock {\em J. Symbolic Logic}, 62(2):337--365, 1997.

\bibitem{Kechris_Martin_II}
A.~S. Kechris and D.~A. Martin.
\newblock On the theory of {$\Pi^{1}_{3}$} sets of reals, {II}.
\newblock In {\em Ordinal Definability and Recursion Theory. {T}he {C}abal
  {S}eminar. {V}olume {III}}, volume~43 of {\em Lect. Notes Log.}, pages
  200--219. Cambridge University Press, Cambridge, 2016.

\bibitem{Q_theory}
Alexander~S. Kechris, Donald~A. Martin, and Robert~M. Solovay.
\newblock Introduction to {$Q$}-theory.
\newblock In {\em Ordinal Definability and Recursion Theory. {T}he {C}abal
  {S}eminar. {V}olume {III}}, volume~43 of {\em Lect. Notes Log.}, pages
  126--199. Cambridge University Press, Cambridge, 2016.

\bibitem{fsit}
William~J. Mitchell and John~R. Steel.
\newblock {\em Fine structure and iteration trees}, volume~3 of {\em Lecture
  Notes in Logic}.
\newblock Springer-Verlag, Berlin, 1994.

\bibitem{mos_dst}
Yiannis~N. Moschovakis.
\newblock {\em Descriptive set theory}, volume 155 of {\em Mathematical Surveys
  and Monographs}.
\newblock American Mathematical Society, Providence, RI, second edition, 2009.

\bibitem{nee_opt_I}
Itay Neeman.
\newblock Optimal proofs of determinacy.
\newblock {\em Bull. Symbolic Logic}, 1(3):327--339, 1995.

\bibitem{nee_opt_II}
Itay Neeman.
\newblock Optimal proofs of determinacy. {II}.
\newblock {\em J. Math. Log.}, 2(2):227--258, 2002.

\bibitem{SUW}
Ralf Schindler, Sandra Uhlenbrock, and Hugh Woodin.
\newblock Mice with finitely many {W}oodin cardinals from optimal determinacy
  hypotheses, available at http://wwwmath.uni-muenster.de/u/rds/.

\bibitem{steel_projective_wo_1995}
J.~R. Steel.
\newblock Projectively well-ordered inner models.
\newblock {\em Ann. Pure Appl. Logic}, 74(1):77--104, 1995.

\bibitem{hod_as_a_core_model}
J.~R. Steel and W.~Hugh Woodin.
\newblock {HOD} as a core model.
\newblock In {\em Ordinal Definability and Recursion Theory. {T}he {C}abal
  {S}eminar. {V}olume {III}}, volume~43 of {\em Lect. Notes Log.}, pages
  257--346. Cambridge University Press, Cambridge, 2016.

\bibitem{steel_hodlr_1995}
John~R. Steel.
\newblock {${\rm HOD}\sp {L({\mathbb{R}})}$} is a core model below {$\Theta$}.
\newblock {\em Bull. Symbolic Logic}, 1(1):75--84, 1995.

\bibitem{steel-handbook}
John~R. Steel.
\newblock An outline of inner model theory.
\newblock In {\em Handbook of set theory. {V}ols. 1, 2, 3}, pages 1595--1684.
  Springer, Dordrecht, 2010.

\bibitem{MR3397344}
Yizheng Zhu.
\newblock Realizing an {${\rm AD}^+$} model as a derived model of a premouse.
\newblock {\em Ann. Pure Appl. Logic}, 166(12):1275--1364, 2015.

\bibitem{sharpI}
Yizheng Zhu.
\newblock The higher sharp, 2016, arXiv:1604.00481.

\end{thebibliography}
\end{document}